\documentclass[12pt]{amsart}
\usepackage{amssymb}
\usepackage{geometry}
\usepackage{amsfonts}
\usepackage{amsmath}
\usepackage{amsthm}
\usepackage{nopageno}
\usepackage{graphicx}
\usepackage{amsfonts,amscd}
\usepackage{verbatim}

\newtheorem{theorem}{Theorem}

\newtheorem{corollary}[theorem]{Corollary}

\newtheorem{lemma}[theorem]{Lemma}

\begin{document}

\author[Hay]{Damon M. Hay}
\address{Department of Mathematics and Statistics \\
    Sam Houston State University
 \\ Lee Drain Building\\Box 2206 \\Huntsville, TX  77341
} \email{dhay@shsu.edu}

\address{Department of Mathematics and Statistics \\
    University of North Florida
 \\ 1 UNF Drive \\Jacksonville, FL  32224
} \email{damon.hay@unf.edu}

\title{Multipliers and hereditary subalgebras of operator algebras}

\begin{abstract}
We generalize some technical results of Glicksberg to the realm of general operator algebras and use them to give a  characterization of open and closed projections in terms of certain multiplier algebras.  This generalizes a theorem of J. Wells characterizing an important class of ideals in uniform algebras.  The difficult implication in our main theorem is that if a projection is open in an operator algebra, then the multiplier algebra of the associated hereditary subalgebra arises as the closure of the subalgebra with respect to the strict topology of the multiplier algebra of a naturally associated hereditary $C^*$-subalgebra.  This immediately implies that the multiplier algebra of an operator algebra $A$ may be obtained as the strict closure of $A$ in the multiplier algebra of the $C^*$-algebra generated by $A$. 
\end{abstract}

\maketitle

\section{Introduction}

By an \emph{operator algebra} we mean a norm-closed subalgebra of the bounded operators on a Hilbert space.  Abstractly, an operator algebra is a norm closed subalgebra of a $C^*$-algebra, and this is how we choose to regard them in this paper.  In any case, there is no assumption about the algebra being closed under the adjoint operation.  The general theory of such `nonselfadjoint' or `general' operator algebras is much less developed than that of $C^*$-algebras.  Many of the results in the general theory of $C^*$-algebras are possible because $C^*$-algebras possess an abundance of positivite elements, whereas general operator algebras need not have any.  Consequently, $C^*$-algebras possess increasing contractive approximate identities, whereas their nonselfadjoint subalgebras do not, necessarily.   Nonetheless, many of the important nonselfadjoint operator algebras do possess contractive approximate identities, though not necessarily increasing, and form a class about which something can be said.

A commutative $C^*$-algebra may always be thought of as a $C_0(\Omega)$-space, the algebra of continuous functions on a locally compact space $\Omega$ which vanish at infinity.  The ordering is obvious, but it is Urysohn's lemma which enables one to really see the existence of approximate identities in the closed ideals of the algebra.  Taking the view that general $C^*$-algebras represent `noncommutative' topological spaces, one could argue that the order structure  and the existence of approximate identities actually `are' the topology.  Indeed, in the general theory of $C^*$-algebras, order arguments and those involving approximate identities abound.  

Thinking of topology in terms of approximate identities can be turned around.  This is precisely what is done when one considers Akemann's open and closed projections in the second dual of a $C^*$-algebra.  More specifically, a projection $p$ in the second dual $A^{**}$ of a $C^*$-algebra $A$ is said to be an \emph{open projection} if it is the weak* limit of an increasing net in $A$.  A projection $q$ in $A^{**}$ is said to be a \emph{closed projection} if $q^{\perp} := 1-q$ is open.   These projections, which are defined using the order structure, allow one to view algebraic objects topologically.  The prototypical example of this is the result which characterizes closed ideals in a $C^*$-algebra as exactly those subspaces supported by an open projection.  

Turning back to the commutative case, a \emph{function algebra} is a norm closed subalgebra of a $C_0(\Omega)$-space which we assume separates the points of $\Omega$.   These are the spaces of which general operator algebras are the noncommutative analogue, and in the study of function algebras one has access, though limited, to the underlying topological space.     For general nonselfadjoint operator algebras, one lacks a robust order structure and the benefits of an underlying topological space.   Since every operator algebra sits inside a $C^*$-algebra, this containing $C^*$-algebra can be viewed as the underlying topological space.  In this view one hopes to do part of the work in the containing $C^*$-algebras and then hope that after passing to the subalgebra things still work out.   Using this approach, some of the general tools from the study of uniform algebras can be made to work in then noncommutative case.  An example of this appears in \cite{Hay} where norm-closed one-sided ideals with contractive approximate identities were characterized in terms of projections in the second dual of the operator algebra which were open  with respect to the containing $C^*$-algebra.  The main result of this paper provides another example of this phenomena.  Here we show that a certain toplogy, namely the `strict' topology, on the multiplier algebra of a $C^*$-algebra passes usefully to a closed subalgebra which generates the $C^*$-algebra.  Consequently, one can obtain the multiplier algebra of an operator algebra in terms of the multiplier algebra of the containing $C^*$-algebra.

For a function algebra $A \subset C_0(\Omega)$, one may define the multiplier algebra $M(A)$ of $A$ to be the space of bounded functions $f$ on $\Omega$, such that $fA \subset A$.   It can be shown that multipliers are necessarily continuous functions, and moreover, that $M(A)$ is a closed unital subalgebra of $ C_b(\Omega)$, the space of all bounded continuous functions on $\Omega$, which is a commutative unital $C^*$-algebra which is also $*$-isomorphic to the $C^*$-algebra of continuous functions on the Stone-\v{C}ech compactification of $\Omega$.  In the case that $A = C_0(\Omega)$, then $M(A)=C_b(\Omega)$.

In addition to the norm topology associated with $M(A)$, we may define the \emph{strict} topology on $M(A)$ to be the locally convex topology generated by seminorms of the form $$x \mapsto \|ax\|$$ for all $x \in M(A)$ and $a \in A$.
It is quite easy to show that $A$ is strictly dense inside $M(A)$.

Suppose $A$ is a \emph{uniform algebra} on a compact Hausdorff space $K$, which is a function algebra containing the constant functions.   If  $J$ is a closed ideal of $A$ posessing a contractive approximate identity, then in \cite{Wells} it is shown, among other things, that the multiplier algebra $M(J)$ of $J$ is equal to the closure of $J$ inside $C_b(K \cap E^C)$ in the strict topology, where $E$ is the closed subset of $K$ on which all elements of $J$ vanish, and where $ E^C$ denotes the set complement of $E$.  The space of all continuous functions on $K \cap E^C$ which vanish at infinity, $C_0(K \cap E^C)$, can be identified with a closed ideal $I$ of $C(K)$, so that the multiplier algebra $M(I)$ of $I$ can be identified with $C_b(K \cap E^C)$.  From this perspective, $J \subset I$ and $M(J)$ equals the closure of $J$ inside $M(I)$ with the strict topology relative to $I$.  Conversely, given a closed set $E$, define $J$ (resp. I) to be the set of functions in $A$ (resp. $C(K)$) which vanish on $E$.  If $M(J)$ is the strict closure of $J$ in $M(I)$, then $J$ is actually a closed ideal possessing a contractive approximate identity.  More precisely, Wells proves the following theorem.

\begin{theorem}[Wells \cite{Wells}]\label{Wellsthm}
Let $E$ be a closed subset of a compact Hausdorff space $K$ and let $A$ be a uniform algebra on $K$.  Denote  by $J$ the ideal of functions in $A$ which vanish on $E$.   The following are equivalent:
\begin{enumerate}
\item For every regular Borel measure $\mu$ on $K$ and $f \in A$, $$\int_K f \; d\mu = 0 \Rightarrow \int_{K \cap E} f \; d\mu = 0 \mbox{ for all } f \in A;$$
\item $M(J)$ is the closure of $J$ inside $C_b(K \cap E^C)$ with respect to the strict topology associated with $C_0(K \cap E^C)$;
\item The constant function $1$ is in the strict closure of $J$ in $C_b(K \cap E^C)$.
\end{enumerate}
\end{theorem}

In this note we consider an analogous result when $A$ is a closed unital subalgebra of a $C^*$-algebra $B$, $J$ is a closed hereditary subalgebra of $A$, and $I$ is an analogously associated hereditary $C^*$-subalgebra of $B$.  See Theorem \ref{thm:intro} below.  

 Multiplier algebras of $C^*$-algebras were introduced by Busby  (\cite{Busby68}) and those of general operator algebras with contractive approximate identities by Poon and Ruan (\cite{PR}).  There are many equivalent definitions of the multiplier algebra, but the one most suitable for our setting and techniques is in terms of the second dual.  Given an operator algebra $A$, selfadjoint or not, with contractive approximate identity, the multiplier algebra $M(A)$ may be taken to be the unital operator algebra
$$\{\eta \in A^{**}: \eta A \subset A \mbox{ and } A \eta \subset A \}.$$  For more information on multipliers of operator algebras (and operator spaces) see the monograph \cite{BL}, and the references therein.

Let $A$ be a unital subalgebra of a unital $C^*$-algebra $B$, with the same unit, and let $p\in A^{**}$ be a projection.   In \cite{BHN}, it is shown that the following are equivalent:
\begin{enumerate}
\item $p \in (pA^{**}p \cap A)^{\perp\perp}$;
\item $p$ is open as a projection in $B^{**}$;
\item $p$ is the left support projection of a right ideal of $A$ with left contractive approximate identity;
\item $p$ is the right support projection of a left ideal of $A$ with right contractive approximate identity.
\end{enumerate}
A projection $p \in A^{**}$ satisfying these equivalent conditions will be said to be \emph{open in $A^{**}$}.  The subalgebra $D = pA^{**}p \cap A$ appearing in condition (1) will be referred to as the \emph{hereditary subalgebra of $A$ supported by $p$}.  Moreover, this `HSA' possesses a contractive approximate identity $(e_t)$.
If we let $C$ be the hereditary $C^*$-subalgebra of $B$ supported by $p$,  that is $C = pB^{**}p \cap B$,  then $D = A \cap C$ and $(e_t)$ is also a contractive approximate identity for $C$.
\begin{lemma} Let $p$, $D$ and  $C$ be as above with $p$ open in $A^{**}$.  Then $M(D) = \{\eta \in M(C): \eta D \subset D \mbox{ and } \eta D \subset D\}$.\end{lemma}

\begin{proof}
Suppose that $\eta \in M(D)$ and $c \in C$.  Then $e_tc \rightarrow c$, where $(e_t)$ is a contractive approximate identity of $D$ and $C$.
Then $\eta c = \lim_t \eta e_t c$, which must lie in $C$, since $\eta e_t \in D \subset C$.  Similarly, $c \eta \in C$, so that $M(D)$ is contained in $\{\eta \in M(C): \eta D \subset D \mbox{ and } \eta D \subset D\}$. On the other hand, if $\nu$ belongs to this set, then $D$ contains the nets $(\nu e_t)$ and $(e_t \nu)$ which converge weak* to $p \nu = \nu = \nu p$, so that $\nu \in M(D)$.
\end{proof}

More generally, suppose that $C$ is given a priori as a hereditary $C^*$-subalgebra of $B$ with support projection $p$ and  $D = A \cap C$.  Then $D$ is still a closed subalgebra
of $A$, but not necessarily a hereditary subalgebra of $A$, nor does it necessarily possess an approximate identity.  Nonetheless, we may still consider the operator algebra $$M(C:D) := \{\eta \in M(C): \eta D \subset D \mbox{ and } \eta D \subset D\},$$  and in the case that $D$ is a hereditary subalgebra of $A$ with support projection $p$, then $M(C:D) = M(D)$.

The \emph{strict topology on $M(C)$} is defined to be the topology generated by seminorms of the form $\eta \mapsto \|\eta b\| + \|c \eta\|$, for $\eta \in M(C)$ and $b,c \in C$.  We denote this topology by $\beta_C$.  In general, the strict topology on $M(C)$ is weaker than the norm topology, and so strictly continuous linear functionals on $M(C)$ are uniformly continuous.

The main result of this paper is the following theorem, which is a noncommutative version Theorem \ref{Wellsthm}.

\begin{theorem}\label{thm:intro}Let $A$ be a unital subalgebra of a $C^*$-algebra $B$ and let $p$ be a projection which is open in $B^{**}$.  Let $C$ be the hereditary $C^*$-subalgebra of
$B$ supported by $p$ and set $D = C \cap A$.  The following are equivalent:
\begin{enumerate}
\item $p$ is open in $A^{**}$;
\item $M(C:D)$ equals the $\beta_C$-closure of $D$ in $M(C)$;
\item $p$ lies in the $\beta_C$-closure of $D$ inside $M(C)$.
\end{enumerate}
Moreover, if any of the above equivalent conditions hold, then $D$ is a hereditary subalgebra of $A$ and $M(C:D) = M(D)$.
\end{theorem}

See Section \ref{main} for the proof.

\begin{corollary}\label{cor:main}
Let $A$ be closed subalgebra of a $C^*$-algebra $B$ such that $A$ generates $B$ as a $C^*$-algebra.  If $A$ has a contractive approximate identity, then $M(A)$ is the $\beta_B$-closure of $A$ inside $M(B)$.
\end{corollary}

\begin{proof}First, $A$ is a hereditary subalgebra of $A^1$, the unitization of $A$.  Since $A$ generates $B$, then the identity of $A^1$ is that of $B^1$, the unitization of $B$, and the HSA of $B$ associated with $1$ is just $B$.  The result follows immediately.
\end{proof}

\begin{corollary}Let $A$ be a unital subalgebra of a $C^*$-algebra $B$ and let $p$ be a projection which is open in $B^{**}$.  If $C$ is the HSA of $B$ supported by $p$ and $D = C \cap A$ has an unbounded or bounded approximate identity which is also an approximate identity for $C$, then $D$ has an approximate identity which is contractive.
\end{corollary}

\begin{proof}If $(e_t)$ is an approximate identity for $D$ which is also such for $C$, then for any $c \in C$, it follows that $e_tc$ and $c e_t$ converge to $pc = c =cp$ in norm.  Hence, $(e_t)$ converges to $p$ in the $\beta_C$-topology, so that $p$ lies in the $\beta_C$-closure of $D$ in $M(C)$.
\end{proof}

\section{Functionals orthogonal to operator algebras}
As before, let $B$ be a unital $C^*$-algebra and let $A \subset B$ be a unital  subalgebra.
We have the following completely isometric homomorphic relations:
$$A \subset A^{**} \cong A^{\perp\perp} \subset B^{**}.$$
We will also use the following lemma frequently and without explicitly mentioning when it is being used.

\begin{lemma}\label{weak}
Let $A$ be a unital subalgebra of a $C^*$-algebra $B$. If $q \in B^{**}$ is a
projection, then the following are equivalent:
\begin{enumerate}
 \item $q \in \overline{A}^{w*}.$
 \item $q \in A^{\perp \perp}.$
 \item $A^{\perp}$ is contained in $(qA)_{\perp}$ (resp. $(Aq)_{\perp}$).
 \item $A^\perp$ is contained in $(qAq)_{\perp}$.
 \item $\mu \in A^{\perp} \Rightarrow \mu q \in A^\perp$ (resp. $q \mu \in A^\perp$).
 \item $\mu \in A^{\perp} \Rightarrow q \mu q\in A^\perp$.
\end{enumerate}
\end{lemma}

\begin{proof}
  The equivalence of (1) and (2) is a standard result of functional
  analysis.  For (2) implies (3), suppose (2) holds.  By
  hypothesis, $\psi(q) = 0$ for all $\psi \in
  A^{\perp}$.  Let $\varphi \in A^{\perp}$.  Then for each $a \in A$,
  $a \varphi  \in A^{\perp}.$  Thus $\varphi(qa) = 0$ for all $a
  \in A$.  Hence $\varphi \in (qA)_{\perp}$, which gives (3).
  For (3) implies (2), we assume that $A^\perp \subset (qA)_{\perp}$, so that $((qA)_{\perp})^{\perp}
  \subset A^{\perp \perp}$.  However, $((qA)_{\perp})^{\perp} =
  \overline{qA}^{w^*} = q \overline{A}^{w^*}$ which must contain $q$
  since $A$ is unital. Hence, (2) holds.
  For (3) implies (4), suppose that $\varphi \in A^\perp$. Then by (3), $\varphi q \in A^\perp$.  By applying the other condition in (3) to $\varphi q$, we then have $q \varphi q \in A^\perp$.  The proof that (4) implies (2) is similar to that of (3) implies (2).
 The equivalence of (3) and (5) is trivial, as well as that of (4) and (6).
\end{proof}

Suppose $J$ is a closed right ideal of $A$ with a left contractive approximate identity $(e_t)$.  For such $J$ there exists an open projection $p$ in $B^{**}$ which is open in $A^{**}$, so that $J = A \cap pA^{**}$.  Moreover, $(e_t)$ converges to $p$ in the weak* topology.  Since $p$ is an open projection in $B^{**}$, it supports a closed right ideal $I = B \cap p B^{**}$ of $B$.  It is easy to see that $I$ contains $J$ and $(e_t)$ is a left contractive approximate identity for $I$.

If $\varphi \in (qB)^*$, then associated to $\varphi$ is element an of
$B^*$ defined by $x \mapsto \varphi(qx)$.  We denote this naturally associated map by
$\varphi q$.
The set $(qB)^*q = \{\varphi q: \varphi \in (qB)^*\}$ is a subspace of $B^*$.  We may also consider the spaces $(Bq)^*$ and $q(Bq)^*$ analogously.

The following lemmas are the noncommutative versions of some results due to Glicksberg (\cite{Glick62}) and are closely related to the subject of peak interpolation for operator algebras.

\begin{lemma}\label{quotientnorm}
Let $q \in B^{**}$ be a closed projection such that $\mu q \in A^\perp$ for every $\mu \in A^\perp$.
Then
$$\|\varphi + (qA)^\perp\| \le \|\varphi q + A^\perp\|,$$
for all $\varphi \in (qB)^*$.
\end{lemma}

\begin{proof}
 Let $\varphi \in (qB)^*$ and  $\psi \in A^{\perp}$.  Then $\psi q \in A^{\perp}$ and we may also regard
$\psi q$ as an element of $(qB)^*$ (by viewing $\psi q$ as acting on $B^{**}$ and then restricting to
the subspace $qB$). Regarding $\psi q$ as such, we also have $\psi q \in (qA)^{\perp} \subset (qB)^*$.
If $I$ is the right ideal in $B$
supported by $q^\perp$, then for $qb \in {\rm  Ball}(qB)$, we have
$\|qb\| = \|b + I\|$.  Since right ideals are
proximinal in a $C^*$-algebra, it follows that there exists $a \in I$
such that $\|qb\|= \|b+I\| = \|b+a\|$.  Since $q(b+a) = qb$ and
$\|b+a\| \le 1$, by replacing $b$ with $b+a$, it
follows that
\begin{equation*}\| \varphi + \psi q \|_{(qB)^*} =  \sup \{ |
  \varphi(qb) + \psi (qb)| : b \in {\rm Ball}(B) \}.
\end{equation*}
However, taking a net $(f_t) \subset \mbox{Ball}(B)$ converging to $q$ weak* with $qf_t = q$,
for $b \in {\rm Ball}(B)$, we have
\begin{eqnarray*} | \varphi(qb) + \psi (qb)| & = & \lim_t |
  \varphi(qf_tb) + \psi (f_tb)| \\
& \le & \|\varphi q + \psi\|_{B^*}.
\end{eqnarray*}
The last inequality follows from the inequality $|\varphi(qf_tb) + \psi(f_tb)| \le  \|\varphi q + \psi\|_{B^*} \|f_tb\| $.
Hence, $\| \varphi + \psi q \|_{(qB)^*} \le \|\varphi q + \psi\|_{B^*},$
and thus,
$\|\varphi + (qA)^{\perp}\| \le \| \varphi q + \psi\|_{B^*}.$
Taking the infimum over all $\psi \in A^\perp$ yields $\|\varphi +
(qA)^{\perp}\| \le \| \varphi q + A^\perp \|$.
\end{proof}

\begin{lemma}\label{tech}Let $q$ be a closed projection in $B^{**}$ such that
$\mu q \in A^\perp$ for all $\mu \in A^\perp$.  If
$J = q^{\perp}A^{**} \cap A$, then $J^{\perp} = A^{\perp} + (qB)^*q$.
\end{lemma}

\begin{proof}
Let $\rho + \mu q \in   A^{\perp} + (qB)^*q$, where $\rho \in A^\perp$ and
$\mu \in (qB)^*$.  Then for $a \in J$, $\rho(a) =0$ and $(\mu q)(a) = \mu(qa) = \mu(0) = 0.$  Thus,
$A^{\perp} + (qB)^*q$ is contained in $J^{\perp}$.

We now show that the weak* closure of $A^{\perp} + (qB)^*q$ is all of $J^{\perp}$.  So suppose that
$b \in (A^{\perp} + (qB)^*q)_{\perp}$.  Then $\rho (b) = 0$ for all $\rho \in A^{\perp}$ and
$\mu(qb) = (\mu q)(b) = 0$ for all $\mu \in (qB)^*$.  Thus, $b \in (A^{\perp})_{\perp} = A$, and $qb = 0$, so that
$b \in J$.  Therefore, $(A^{\perp} + (qB)^*q)_{\perp} \subset J$, from which it follows that
$A^{\perp} + (qB)^*q$ is weak* dense in $J^{\perp}$.

To finish the proof it now suffices to show that $A^{\perp} + (qB)^*q$ is weak* closed.  By Goldstine's
theorem, we only need to show that the unit ball of $A^{\perp} + (qB)^*q$ is weak* closed.  Suppose that
$\varphi$ is the weak* limit of a net $(\mu_t+\nu_t q)$ in the unit ball of $A^{\perp} + (qB)^*q$, where
$\mu_t \in A^\perp$ and $\nu_t \in (qB)^*$. By Lemma \ref{quotientnorm},
$$\|\nu_t + (qA)^\perp\| \le \| \nu_tq + A^\perp\|.$$
Therefore, $\|\nu_t + (qA)^{\perp}\| \le \|\nu_tq + \mu_t\|$, and there exists
$\psi_t \in (qA)^{\perp}$ such that
$$\|\nu_t - \psi_t\| < \|\nu_t + (qA)^{\perp}\| + 1 \le \|\nu_tq + \mu_t\| + 1 \le 2.$$
By weak* compactness, $(\nu_t - \psi_t)$ has a limit point $\nu$ in $(qB)^*$.  However,
$$\mu_t + \psi_tq = (\mu_t+\nu_tq)-(\nu_tq-\psi_tq) = (\mu_t+\nu_tq)-(\nu_t-\psi_t)q,$$
so that a subnet of $(\mu_t + \psi_t q)$
converges to $\varphi-\nu q$.  Since $\psi_t \in (qA)^\perp$, it follows that
$\psi_t q \in A^\perp$, so that $\varphi-\nu q$ lies in $A^\perp$.  Therefore,
$\varphi \in A^\perp + (qB)^*$, which shows that $A^\perp + (qB)^*$ is weak* closed.
\end{proof}

The following is an HSA variant of the previous lemma.

\begin{lemma}\label{hsatech}Let $q$ be a closed projection in $B^{**}$ such that
$\mu q \in A^\perp$ for all $\mu \in A^\perp$.  If
$D = q^{\perp}A^{**}q^{\perp} \cap A$, then $D^{\perp} = A^{\perp} + (qB)^*q + q(Bq)^*$.
\end{lemma}

\begin{proof}Set $p = q^\perp$, and let $J = pA^{**} \cap A$ and $K = A^{**}p \cap A$.  It is clear that $K^\perp + J^\perp \subset D^\perp$.  We show the other containment.  We argue as in the proof of Theorem 2.4 of \cite{BHN} to get that $(qA^* + A^*q)^\perp = pA^{**}p$.  Since $qA^* + A^*q$ is shown to be weak*-closed there, it follows that $qA^* + A^*q = (pA^{**}p)_\perp$.  If $\varphi \in D^\perp$, then by a weak* continuity argument, it follows that $\varphi \in (pA^{**}p)_\perp$.  Thus, there exist $\psi_1, \psi_2 \in A^*$ such that $\varphi = q\psi_1 + \psi_2q$.  Moreover, after extending $\psi_1$ and $\psi_2$ to $B$, it is easy to see that $q\psi_1$ and  $\psi_2q$ lie in $K^\perp$ and $J^\perp$, respectively.  So $D^\perp = K^\perp + J^\perp$.  By Lemma \ref{tech} and a left-handed variant of it, we have $J^{\perp} = A^{\perp} + (qB)^*q$ and $K^{\perp} = A^{\perp} + q(Bq)^*$.  Putting these two together gives the desired result.
\end{proof}

We close this section by stating a variant of a theorem due to D.C. Taylor which is needed in the next section.  It relates strictly continuous linear functionals on the multiplier algebra to bounded linear functionals.

\begin{theorem}\label{Tay}(Taylor \cite{Tay70})If $C$ is a $C^*$-algebra, then the ${\beta_C}$-continuous linear functionals on $M(C)$ are precisely those functionals on $M(C)$ of the form $c \mu b$ for $b,c \in C$ and $\mu \in M(C)^*$.
\end{theorem}

Here we are using the natural left and right module actions of $C$ on $M(C)^*$.  For example, given $\mu \in C^*$ and $c \in C$, then $(c \mu)(x) = \mu (xc)$ for all $x \in M(C)$.

\section{Main result}\label{main}
\begin{theorem}Let $A$ be a unital subalgebra of a $C^*$-algebra $B$ and let $p$ be an open projection in $B^{**}$.  If $C$ is the hereditary subalgebra of $B$ supported by $p$ and $D = A \cap C$, then the following are equivalent:
\begin{enumerate}
\item $p$ is open in $A^{**}$.
\item $M(C:D)$ equals the $\beta_C$-closure of $D$ in $M(C)$.
\item $p$ lies in the $\beta_C$-closure of $D$ inside $M(C)$.
\end{enumerate}
Moreover, if any of the above equivalent conditions hold, then $D$ is a hereditary subalgebra of $A$ and $M(C:D) = M(D)$.
\end{theorem}

\begin{proof}
For $(1)$ implies (2), we first note that (1) implies that $p \in D^{\perp\perp}$
so that $D$ has a contractive approximate identity.  Thus, $M(D) = M(C:D)$.
Let $\eta$ be in the $\beta_C$-closure of $D \subset M(C)$ and suppose that $(a_t)$ is a net in
$D$ converging $\beta_C$-strictly to $\eta$.  Then for $b \in D \subset C$, we have
that $a_t b$ converges to $\eta b$ in norm.
However, $a_t b$ is an element of $D$, which is norm closed.
Hence, $\eta b \in D$, and a similar argument shows that $b \eta \in D$, so that $\eta \in M(C:D)$.
Thus, $\overline{D}^{\beta_C} \subset M(C:D)$.

Now suppose that $\eta$ is in $M(C:D)$ but not in $\overline{D}^{\beta_C}$.  Then there is a $\beta_C$-continuous linear functional $\varphi$ on $M(C)$ such that $\varphi (\eta) = 1$ and $\varphi (D) = 0$.  However, by the result of Taylor above, there exist $b_0, c_0 \in C$ and $\lambda_0 \in M(C)^*$ such that $\varphi = c_0 \lambda_0 b_0 $.
Define $\rho \in B^*$ by  $\rho (x) = \lambda_0(b_0 x \eta c_0)$ for all $x \in B$.
To see that this is well-defined, let $x \in B$.   Since $\eta$ is a multiplier of $C$, we have $\eta c_0 \in C$, and because $C$ is an inner ideal of $B$, it follows that $b_0 x \eta c_0 \in C$.
Moreover, we also have
$$\rho (1) = \lambda_0(b_0 1 \eta c_0) = \varphi (\eta) = 1,$$
and for $a \in D$, we have
$$\rho(a) =  \lambda_0(b_0 a \eta c_0) = \varphi (a \eta),$$ since $a \eta \in
D$.   Thus, $\rho \in D^{\perp}$.  Set $q = p^\perp$.  From the definition of $\rho$ and the fact that $b_0q = qc_0 = 0$, it is also clear that $q \rho q = q \rho = \rho q = 0$.

 By Lemma \ref{hsatech} there exists $\mu \in A^{\perp}$ and $\psi_1 \in (q B)^{*}$ and $\psi_2 \in (Bq)^*$ such that
$\rho = \mu + \psi_1 q + q \psi_2$.  Since $p + q = 1$, we may write
\begin{eqnarray*}\mu + \psi_1 q + q \psi_2 & = & \mu + p \psi_1 q +q \psi_1 q + q \psi_2 p + q \psi_2 q \\
& = &\mu + p \psi_1 q + q \psi_2 p + (q \psi_1 q +  q \psi_2 q ).
\end{eqnarray*}
Since we are assuming $(1)$ holds, we have $\mu q, q \mu, q \mu q \in A^{\perp}$.
Pre- and post-multiplying $\rho$ by $q$ leads to
$$0 = q \rho q = q \mu q + (q \psi_1 q +  q \psi_2 q ),$$ so that $q \psi_1 q +  q \psi_2 q \in A^\perp.$  Only pre-multiplying $\rho$ by $q$ leads to
$$0 = q \rho = q \mu + q \psi_2 p + (q \psi_1 q +  q \psi_2 q ),$$ so that $q \psi_2 p \in A^\perp.$  Finally, post-multiplying by $q$
gives
$$0 = \rho q = p \psi_1 q + (q \psi_1 q +  q \psi_2 q ),$$ so that $p \psi_1 q \in A^\perp.$
We conclude that $\rho \in A^{\perp}$, which contradicts $\rho(1) = 1$.

($(2) \Rightarrow (3)$) is obvious.

($(3) \Rightarrow (1)$) It suffices to show that $p \in D^{\perp \perp}$.  Suppose there exists a net $(a_t)$ in $D$ which converges $\beta_C$-strictly to $p$.  If $p$ is not in $D^{\perp\perp}$, there exists a $\varphi \in B^*$ such that $\varphi (p) =1$, yet $\varphi (D) = 0$.  Restricting $\varphi$ to $C$, there exists $c_0 \in C$ and $\lambda_0 \in C^*$ such that $\varphi |_C = c_0\lambda_0$.  Then
$$0 = \varphi (a_t) = \lambda_0(a_tc_0),$$ which converges to $\lambda_0(pc_0) = \lambda_0(c_0)$, so that $\lambda_0 (c_0) = 0.$  If $(e_t)$ is a contractive approximate identity for $C$, then $\varphi (p) = \lim_s \varphi (e_s) = \lim_s\lambda_0(e_sc_0) = \lim_s\lambda_0(c_0)= 0$, a contradiction.
\end{proof}

\end{document}